\documentclass[11pt]{article}
\usepackage[a4paper,margin=3cm]{geometry}
\usepackage[utf8]{inputenc} 
\usepackage[T1]{fontenc} 
\usepackage{lmodern}
\usepackage[autolanguage]{numprint} 
\usepackage{hyperref} 
\usepackage{graphicx} 
\usepackage[all]{xy}
\usepackage{amsmath,amssymb,amsthm}

\newcommand\NN{\mathbb{N}} 
\newcommand\RR{\mathbb{R}} 
\newcommand\CC{\mathbb{C}} 
\newcommand\PP{\mathbb{P}} 
\newcommand\EE{\mathbb{E}} 

\newtheorem{theorem}{Theorem}[section]%
\newtheorem{lemma}[theorem]{Lemma}%

\author{Rapha\"{e}l Butez\footnote{CEREMADE, UMR CNRS 7534 Universit\'{e} Paris-Dauphine, PSL Research university, Place du Mar\'{e}chal de Lattre de Tassigny 75016 Paris, FRANCE.
Partially supported
by an ANR grant as part of the program "Investissements Avenir" 
ANR-10-LABX-0098 supported by the Fondation Sciences Math\'{e}matiques de Paris. }, Ofer Zeitouni\footnote{Department of Mathematics, Weizmann Institute of Science, POB 26, Rehovot 76100, Israel
and  Courant Institute, New York University, 251 Mercer Street, New York, NY 10012, USA. Partially supported by an Israel Science Foundation grant.}}

\title{Universal large deviations for Kac polynomials}

\begin{document} 

\maketitle 

\begin{abstract}
We prove the universality of the large deviations principle for the 
empirical measures of zeros of 
random polynomials whose coefficients are i.i.d. random variables possessing 
a density with respect to the Lebesgue measure on $\CC$, $\RR$ or $\RR^+$, 
under the assumption that the density does not vanish too fast at zero and decays at least as $\exp -|x|^{\rho}$, $\rho>0$, at infinity.
\end{abstract}

\section{Introduction and statement of the main result}
Consider random polynomials of the form:
\begin{equation}
	\label{eq-poly}
	P_n(z)= \sum_{k=0}^{n}a_k z^k= a_n\prod_{i=1}^n (z-z_i)
\end{equation}
where $a_0, \dots, a_n$ are i.i.d. random variables and $z_1,\ldots,z_n$
are the complex zeros of $P_n$.  (Such random polynomials are often referred to
as \textit{Kac polynomials}.)
There is a rich literature about the behavior of the zeros of $P_n$ 
and we refer to \cite{taovu} for a nice recent review of the subject. 
An important aspect of the theory is universality. For example, 
introduce the empirical measure of zeros:
\begin{equation*}
\mu_n := \frac{1}{n} \sum_{k=1}^{n} \delta_{z_k}.
\end{equation*}
Then,
 Ibragimov and Zaporozhets in \cite{ibragimovzaporozhets} showed that
$(\mu_n)_{n \in \NN}$ converges weakly
to the $\nu_{S^1}$, the uniform measure on the unit circle,
if and only if $\EE(\log(1+|a_0|))< \infty$; that is, the limit $\mu_n$ is 
(modulus technical conditions) universal. Other universal properties include
rescaled limits for $\mu_n$, see
\cite{ibragimovzeitouni}, correlation functions for the point process of zeros \cite{taovu}, and more.

Our interest in this note is in large deviations for the sequence $\mu_n$ in the
space 
$\mathcal{M}_1(\CC)$ equipped with the topology of weak convergence, which makes
it into a Polish space. In case the coefficients $(a_i)$ are i.i.d. standard
complex random variables, Zeitouni and Zelditch\footnote{In fact, \cite{zeitounizelditch} work in $\mathbb{CP}^1$ and consider more general 
ensembles of holomorphic polynomial with Gaussian coefficients, but it is not
hard
to check that their result, when specialized to Kac polynomials with complex Gaussian coefficients, is equivalent to the one here; this is implicitly stated in
\cite{zeitounizelditch} and explicitely
checked in \cite{butez}.}
proved in \cite{zeitounizelditch} that the sequence of empirical measures of zeros 
(which we denote by 
$\mu_n^{\CC}$ for this model)
satisfies the large deviations principle (LDP) in $\mathcal{M}_1(\CC)$ 
with speed $n^2$ and good rate function $I_{\CC}$ defined by
\begin{align*}
I_{\CC}(\mu) & =  -\iint \left( \log|z-w|-\frac{1}{2}\log(1+|z|^2)-\frac{1}{2}\log(1+|w|^2)\right)d\mu(z)d\mu(w) \\ 
& \qquad \qquad +  \sup_{z \in S^1} \int \left( \log|z-w|^2-\log(1+|w|^2)\right)d\mu(w).
\end{align*}
When $\displaystyle \int \log(1+|z|^2)d\mu(z)$ is finite, 
it simplifies to:
\[I_{\CC}(\mu)=  -\iint \log|z-w|d\mu(z)d\mu(w) + \sup_{z \in S^1} \int \log|z-w|^2d\mu(w).\] 
This has been extended by Butez \cite{butez}
to the case of real-valued i.i.d. standard
Gaussians $(a_i)$:
the empirical measure of zeros, denoted 
$\mu_n^{\RR}$ for that model, 
satisfies the LDP  in $\mathcal{M}_1(\CC)$ 
with speed $n^2$ and good rate function $I_{\RR}$ defined by
\[I_{\RR}(\mu)= \begin{cases}
 \frac{1}{2} I_{\CC}(\mu) &\text{if $\mu$ is invariant under } z\mapsto \bar{z} \\
 \infty &\text{otherwise.}
\end{cases}
\]
Finally, when
the coefficients $(a_i)$ are i.i.d. standard exponential
random variables, Ghosh and Zeitouni proved in \cite{ghoshzeitouni} 
that the sequence of empirical measures  of zeros, denoted by
$\mu_n^{\RR^+}$, 
satisfies the  LDP  in 
$\mathcal{M}_1(\CC)$ with speed $n^2$ and good rate function 
$I_{\RR^+}$ defined by:
\[I_{\RR^+}(\mu)= \begin{cases}
 \frac{1}{2} I_{\CC}(\mu) & \text{ if } \mu \in \bar{\mathcal{P}} \\
  \infty & \text{ otherwise.}
\end{cases}
\]
where $\mathcal{P}$ is the set of empirical measures of zeros
of polynomials with positive coefficients and 
$\bar{\mathcal{P}}$ is its closure for the weak topology.
(An explicit description of $\mathcal{P}$ is provided in \cite{bergweilereremenko}.)

Apart for the models described above, 
to our knowledge no other LDPs for the
empirical measure of zeros of Kac polynomials appear in the literature.
(In a different direction, Zelditch \cite{zelditch} extended the results
of \cite{zeitounizelditch} to the case of Riemann surfaces, and
Feng and Zelditch \cite{feng} 
studied 
some cases of non-i.i.d. 
coefficients 
in the context of
more general $P(\phi)_2$ random polynomials.)  

Our main result concerns the universality of the above large deviation
principles, under mild technical conditions.
\begin{theorem}\label{maintheorem}
Let $E$ be $\CC$, $\RR$ or $\RR^+$. Let $a_0, \dots, a_n$ be i.i.d. 
random variables with a density $g$ with respect to the Lebesgue 
measure on $E$. Assume that:
\begin{enumerate}
	\item
There exist $\rho >0$, $r>0$ and $R>0$ such that 
\begin{equation}\label{sup}
\forall z \in \CC, \quad g(z) \leq \exp(-r|z|^{\rho} + R),
\end{equation}
\item
There exits $\delta>0$ such that for all $\lambda >0$: 
\begin{equation}\label{inf}
\int 1_{|x|\leq \delta}\frac{1}{g(x)^{\lambda}}d\ell_E(x) < \infty
\end{equation}
\end{enumerate}
Then the sequence of empirical measures $(\mu_n)_{n \in \NN}$ satisfies a large deviations principle in $\mathcal{M}_1(\CC)$ 
with speed $n^2$ and rate function $I_{E}$.
\end{theorem}
The second assumption in Theorem 
\ref{maintheorem}
means that either the density $g$ does not vanish around zero or, 
if it vanishes, $g$ is greater than any $|x|^a$ in a neighborhood of zero.

Before describing the (simple)
ideas behind the proof of Theorem \ref{maintheorem}, we
explain some of the background and why we find the theorem somewhat surprising.
The proof of the LDPs  in the Gaussian and Exponential 
cases is based on an explicit expression for the joint distributions of zeros, 
that we review below. Given that expression,
the proofs of the LDP follow (with some detours) 
a track well explored in the case of the empirical measure of eigenvalues of
random matrices.
For the latter,
large deviations have been extensively studied, initially by Ben Arous and Guionnet \cite{benarousguionnet}, Ben Arous and Zeitouni \cite{benarouszeitouni} 
and Hiai and Petz \cite{hiaipetz}. Recently, the large deviations for the 
empirical measure were proved for Wigner matrices with entries possessing heavier-than-Gaussian tails 
by Bordenave and Caputo \cite{bordenavecaputo}, with a rate function 
depending on the tail of the entries. Very similar results were obtained
by Groux \cite{groux} for Wishart matrices.
In particular, it follows from these results that
in the random matrix setup, the rate function is known to
not be universal; this is in sharp contrast with Theorem 
\ref{maintheorem}.

As mentioned above, the LDP for Kac polynomials in the Gaussian
and exponential coefficients cases begin with an explicit expression 
for the density of zeros, which we now explain. We concentrate 
first on the case of complex Gaussian coefficients. Note that
the second equality in \eqref{eq-poly} gives an $n!$-to-1 map between
$(a_n,z_1,\ldots,z_n)$ and
$(a_0,\ldots,a_n)$.  A classical computation of the Jacobian followed by integration over $a_n$, see
e.g. 
\cite{bogomolny}, \cite{butez},\cite{forrester}, \cite{zeitounizelditch},
shows that the distribution of $(z_1,\dots,z_n)$ possesses 
a density with respect to the Lebesgue measure
$d\ell_{\CC^n}$ on $\CC^n$
given by:
\begin{equation} \label{gaussiencomplexe}
\frac{n!}{\pi^n} \frac{\prod_{i<j}|z_i-z_j|^2}{\left(\displaystyle \int \prod_{k=1}^{n}|z-z_k|^2 d\nu_{S^1}(z)\right)^{n+1}}
= \frac{n!}{\pi^n} \frac{\prod_{i<j}|z_i-z_j|^2}{\|\tilde{a}\|_2^{2n+2}} 
\end{equation}
where $\tilde{a}=(a_0/a_n,\dots,a_{n-1}/a_n,1)$ is a continuous 
function of $(z_1,\ldots,z_n)$ given explicitely by Vieta's formula. 

In the case of real Gaussian coefficients,
the probability of having $k$ real zeros is positive for $k$ having the same
parity of $n$. Following
a computation  of Zaporozhet in \cite{zaporozhets}, one obtains 
that the distribution of the roots of $P_n$ is given by:
\begin{align*}
\sum_{k=0}^{\lfloor n/2 \rfloor} \frac{2^k \Gamma(\frac{n+1}{2})}{k!(n-2k)!\pi^{(n-1)/2}} & \frac{ \prod_{i<j}|z_i-z_j|}{ (\int\prod_{i=1}^n |z-z_i|^2 d\nu_{S^1})^{(n+1)/2}} d\ell_{n,k}(z_1,\dots,z_n) \\ 
&\quad =  \sum_{k=0}^{\lfloor n/2 \rfloor} \frac{2^k \Gamma(\frac{n+1}{2})}{k!(n-2k)!\pi^{(n-1)/2}}  \frac{ \prod_{i<j}|z_i-z_j|}{ \| \tilde{a}\|_2^{n+1}} d\ell_{n,k}(z_1,\dots,z_n)
\end{align*}
where \begin{equation}
d\ell_{n,k}(z_1,\dots,z_n)= d\ell_{\RR}(z_1) \dots   d\ell_{\RR}(z_{n-2k}) d\ell_{\CC}(z_{n-k})  \dots  d\ell_{\CC}(z_n).
\end{equation}
Note that the $k$-th term of the mixture corresponds to the case where $P_n$ has $n-2k$ real roots.  

Finally, in the case of positive exponential real coefficients,
the distribution of the vector of the zeros is given by:
\begin{align*}
\sum_{k=0}^{\lfloor n/2 \rfloor} \frac{2^k n!}{k!(n-2k)!}& \frac{ \prod_{i<j}|z_i-z_j|}{ (\prod_{i=1}^n |1-z_i|)^{(n+1)}} d\ell_{n,k}(z_1,\dots,z_n) \\
& = \sum_{k=0}^{\lfloor n/2 \rfloor} \frac{2^k n!}{k!(n-2k)!} \frac{ \prod_{i<j}|z_i-z_j|}{ \|\tilde{a}\|_1^{n+1}} d\ell_{n,k}(z_1,\dots,z_n).
\end{align*}

\noindent \textbf{Main idea of the proof of Theorem \ref{maintheorem}.}
We will prove the universality of the LDP
by comparing the distributions of the vectors of the zeros in the different models. 
Assume one could
find two sequences $(b_n)_{n \in \NN}$ and $(c_n)_{n \in \NN}$ satisfying 
\begin{equation}
\lim_{n \to \infty} \frac{1}{n^2}\log b_n = \lim_{n \to \infty} \frac{1}{n^2} \log c_n = 0 
\end{equation}
and two probability densities on $\CC^{n}$, $F_n$ and $G_n$ satisfying:
\begin{equation}\label{encadrement}
\forall (z_1, \dots, z_n)\in \CC^{n} \qquad b_nF_n(z_1, \dots, z_n) \leq p(z_1,\dots,z_n) \leq c_n G_n(z_1,\dots,z_n)
\end{equation}
such that, under the distribution given by $F_n$ or $G_n$, 
the sequence of empirical measures $(\mu_n^{F_n})_{n\in \NN}$ and 
$(\mu_n^{G_n})_{n \in \NN}$ satisfies a LDP
in $\mathcal{M}_1(\CC)$, with speed $n^2$ and the same rate function $I$.
Then, the sequence $(\mu_n)_{n \in \NN}$ satisfies a LDP
in $\mathcal{M}_1(\CC)$ with speed $n^2$ and rate function $I$, since
for any set $A$, $\PP(\mu_n \in A)$ is an integral with respect 
to the distribution of the zeros and therefore
one can use the 
bounds \eqref{encadrement} to obtain the LDP.

In practice, we will obtain \eqref{encadrement} by noting that
if the joint distribution of the coefficients is a function of a norm 
$\|.\|$ of the vector of the coefficients, 
the distribution of the zeros is roughly of the form:
\[ \frac{\prod_{i<j}|z_i-z_j|^2}{\| \tilde{a} \|^{2n+2}} .\] 
If $\|. \|$ can be compared with $\|.\|_2$ with nice constants, 
we can can relate the density of zeros by one that is closer 
to the Gaussian or exponential cases in the spirit of \eqref{encadrement}.
For i.i.d. variables, the first hypothesis of the theorem is used 
to replace the joint distribution of the coefficients 
by a function of $\| a \|_{\rho}$ 
and then we prove the upper bound for the latter distribution. 
The second hypothesis means that, for the lower bound, we can replace 
the joint distribution of $a$ by a $1_{\|a \|_{\infty} \leq \delta}$ 
which is also a function of a norm.

We conclude this introduction by stating and proving 
a technical lemma that will
be used in the proof of the LDP lower bound.
\begin{lemma}\label{technicalLemma} Let $E$ be $\CC$, $\RR$ or $\RR^+$. 
Let $X_0, \dots, X_n$ be i.i.d. random variables, 
uniformly distributed on the disk of center $0$ and radius $\delta$ of $E$. 
Assume that there exits $\delta>0$ such that for all 
$\lambda >0$, \begin{equation}
	\label{eq-clambda}
c(\lambda):=\int 1_{|x|\leq \delta}\frac{1}{g(x)^{\lambda}}d\ell_E(x) < \infty.
\end{equation}
Then, for any $K>0$ and $\varepsilon>0$
there exists $n_0=n_0(K,\delta,\varepsilon)$ such that for all $n>n_0$,
\begin{align}
     \int 1_{\{\prod_{k=0}^{n}g(a_i)\leq e^{-\varepsilon n^2}\}}   1_{\| a \|_{\infty}< \delta} d\ell_{E^n}(a_0,\dots,a_n)  \leq e^{-Kn^2}.
\end{align}
\end{lemma}

\begin{proof}[Proof of Lemma \ref{technicalLemma}]
	Fix $K>0$, micking the proof of Chernoff's inequality, we have:
\begin{align*}
\int  1_{\{\prod_{k=0}^{n}  g(a_i)\leq e^{-\varepsilon n^2}\}}  &1_{\| a \|_{\infty}< \delta}  d\ell_{E^{n+1}}(a_0,\dots,a_n) 
\\ &= \int  1_{\{\prod_{k=0}^{n}  g(a_i)^{-\lambda}\geq e^{\lambda\varepsilon n^2}\}}  1_{\| a \|_{\infty}< \delta}  d\ell_{E^{n+1}}(a_0,\dots,a_n)
\\ & \leq e^{-\lambda \varepsilon n^2}\int \prod_{k=0}^{n}\left[g(a_k)^{-\lambda} 1_{|a_k|<\delta}\right] d\ell_{E^{n+1}}(a_0,\dots,a_n)
\\ & \leq e^{-\lambda \varepsilon n^2} e^{(n+1) c(\lambda)}.
\end{align*}
The proof is completed 
by taking $\lambda$ large enough so that $\lambda \varepsilon>K$ and then taking $n_0$ large enough so that
 $n! e^{-\lambda \varepsilon n^2} e^{(n+1) c(\lambda)} \leq e^{-Kn^2}$ for all $n>n_0$.
\end{proof}

\section{Proof of Theorem \ref{maintheorem}}
The proof of the main theorem is made in two steps: we start by 
proving the theorem when the coefficients are complex, 
and then we treat the real and the positive case. 
The proof of the three cases are very similar, 
the arguments and ideas are exactly the same.
\begin{proof}[Proof of Theorem \ref{maintheorem}]
\textbf{Complex coefficients.} \newline
Recall that the density of 
the distribution of
the random vector of zeros $(z_1,\dots,z_n)$ (taken at random uniform order)
with respect to
$\ell_{\CC^{n}}$
is given by
\[p(z_1,\dots,z_n)=\int \prod_{i<j}|z_i-z_j|^2 |a_n|^{2n} 
g(a_0)\dots g(a_n)  d\ell_{\CC}(a_n)
\]
where the $a_j$'s are seen as functions of $z_1$, \dots, $z_n$ and $a_n$ using
Vieta's formula. 
See e.g. \cite[Lemma 1.1.1 p3]{houghtkris} for a proof of this classical result. 
\newline
\textbf{Upper Bound.}
Using the inequality \eqref{sup}, we obtain:
\begin{align*}
p(z_1,\dots,z_n) & \leq \int \prod_{i<j}|z_i-z_j|^2 |a_n|^{2n} \exp(-r \sum_{k=0}^{n} |a_k|^{\rho}) e^{(n+1)R} d\ell_{\CC}(a_n) 
\end{align*}
For a vector $b=(b_0,\ldots,b_n)$ and $\rho>0$, set $\|b\|_\rho= 
(\sum_{i=0}^n |b_i|^\rho)^{1/\rho}$.
Then,
\begin{equation*}
\int \prod_{i<j}|z_i-z_j|^2 |a_n|^{2n} \exp(-r \sum_{k=0}^{n} |a_k|^{\rho})d\ell_{\CC}(a_n) = \int \prod_{i<j}|z_i-z_j|^2 |a_n|^{2n} \exp(-r |a_n|^{\rho}\|\tilde{a}\|_{\rho}^{\rho})d\ell_{\CC}(a_n)
\end{equation*}
where $\tilde{a}=(a_0/a_n, \dots, a_{n-1}/a_n,1)$. We note
that $\tilde{a}$ only depends on the zeros and not on $a_n$, 
so we can compute the last
integral using the change of variables $u=a_n \|\tilde{a}\|_{\rho}$ to obtain:
\begin{align*}
	p(z_1,\dots,z_n) \leq 
	e^{(n+1)R}
	\frac{\prod_{i<j}|z_i-z_j|^2}{\|\tilde{a}\|_{\rho}^{2n+2}} \int |u|^{2n} e^{-r |u|^{\rho}} d\ell_{\CC}(u).
\end{align*}
Finally, using the classical inequalities on $\CC^{n+1}$:
\begin{align*}
 \text{if } \rho > 2,& \quad \| . \|_{\rho} \geq \frac{1}{n^{1/2-1/\rho}} \| . \|_2 \\
 \text{if } \rho \leq 2, & \quad \| . \|_{\rho} \geq  \| . \|_2
\end{align*}
we obtain that there exists a sequence $(\gamma_n)_{n \in \NN}$ such that, for
any $\rho>0$,
\begin{equation}\label{gamma}
 \| . \|_{\rho} \geq \gamma_n \| . \|_2 \quad \text{and } \quad \lim_{n \to \infty}\frac{1}{n}\log \gamma_n = 0.
\end{equation}
Using this inequality we get
\begin{align*}
p(z_1,\dots,z_n)& \leq 
	e^{(n+1)R}
\frac{\prod_{i<j}|z_i-z_j|^2}{\|\tilde{a}\|_{2}^{2n+2}} \frac{1}{\gamma_n^{2n+2}} \int |u|^{2n} e^{-r |u|^{\rho}} d\ell_{\CC}(u)  \\
& \leq 
	e^{(n+1)R}
\frac{n!}{\pi^{n+1}}\frac{\prod_{i<j}|z_i-z_j|^2}{\|\tilde{a}\|_{2}^{2n+2}} \times \frac{\pi^{n+1}}{n! \gamma_n^{2n+2}} \int |u|^{2n} e^{-r |u|^{\rho}} d\ell_{\CC}(u) .
\end{align*}
The first term of the product is the distribution 
$\mu_n^{\CC}$, see
\eqref{gaussiencomplexe}, and thanks to \eqref{gamma} we have
\[c_n:=
	e^{(n+1)R}
\frac{\pi^{n+1} }{n!\gamma_n^{2n+2}} \int |u|^{2n} e^{-r |u|^{\rho}} d\ell_{\CC}(u)=e^{O(n\log n)} . \]

Let $A \subset \mathcal{M}_1(\CC)$ be a Borel set. 
Then,
\begin{align*}
\frac{1}{n^2}\log \PP(\mu_n \in A) =& \frac{1}{n^2}\log \int 1_{\mu_n \in A} p(z_1,\dots, z_n)d\ell_{\CC^n}(z_1,\dots,z_n) \\ \leq&  \frac{1}{n^2}\log \int 1_{\mu_n \in A} \frac{n!}{\pi^{n+1}} \frac{\prod_{i<j}|z_i-z_j|^2}{\left( \int \prod_{k=1}^{n}|z-z_k|^2 d\nu_{S^1}(z)\right)^{n+1}} d\ell_{\CC^n}(z_1, \dots , z_n) + \frac{\log c_n}{n^2}\\
= & \frac{1}{n^2}\log \PP(\mu_n^{\CC} \in A) + \frac{\log c_n}{n^2}.
\end{align*}
Therefore, using the LDP upper bound for $\mu_n^{\CC}$, we complete the proof
of the upper bound by noting that 
\[
	\limsup_{n \to \infty} \frac{1}{n^2}\log \PP(\mu_n \in A) \leq - 
	\inf_{\textbf{clo}A} I_\CC.
\]
\newline
\textbf{Lower Bound.} First, we show that the technical lemma allows us to reduce the problem to the proof of the lower bound for i.i.d. $(a_i)$, with uniform distribution on the disk $D(0,\delta)$.
Let $A\subset \mathcal{M}_1(\CC)$ be a Borel set with
$\inf_{\mathrm{int} A} I_{\CC} < \infty$, fix
$K > \inf_{\mathrm{int} A} I_{\CC} $ and $\varepsilon>0$ then, 
thanks to Lemma \ref{technicalLemma} there exists $n_0$ such that for any $n>n_0$: 
\begin{align} \label{eq-align}
\PP(\mu_n \in A) & = \int 1_{\mu_n \in A} \prod_{k=0}^n g(a_k) d\ell_{\CC}(a_0)\dots d\ell_{\CC}(a_n) \nonumber\\
& \geq \int 1_{\{\prod_{k=0}^{n}g(a_i)\geq e^{-\varepsilon n^2}\}}1_{\mu_n \in A}1_{\| a \|_{\infty}< \delta} \prod_{k=0}^n g(a_k) d\ell_{\CC}(a_0)\dots d\ell_{\CC}(a_n) \nonumber\\
& \geq e^{-\varepsilon n^2} \int 1_{\{\prod_{k=0}^{n}g(a_i)\geq e^{-\varepsilon n^2}\}}1_{\mu_n \in A}1_{\| a \|_{\infty}< \delta} d\ell_{\CC}(a_0)\dots d\ell_{\CC}(a_n) \nonumber\\
& \geq  e^{-\varepsilon n^2} \int 1_{\mu_n \in A}1_{\| a \|_{\infty}< \delta}d\ell_{\CC}(a_0)\dots d\ell_{\CC}(a_n) - e^{-(K+\varepsilon)n^2}. 
\end{align}
The integral $\int 1_{\mu_n \in A}1_{\| a \|_{\infty}< \delta}d\ell_{\CC}(a_0)\dots d\ell_{\CC}(a_n)$ is, up to a normalizing factor $(\pi \delta^2)^{n+1}$ which is of order $e^{O(n)}$, the probability that the empirical mesure of the zeros of a random polynomial with i.i.d. uniform coefficients on the disk $D(0,\delta)$ belongs in $A$.

\noindent Now we deal with this integral using the same techniques used for the upper bound:
\begin{align*}
\int 1_{\mu_n \in A}1_{\| a \|_{\infty}< \delta}&d\ell_{\CC}(a_0)\dots d\ell_{\CC}(a_n)   \\&= \int 1_{\mu_n \in A}\prod_{i<j}|z_i-z_j|^2 \int 1_{|a_n|\| \tilde{a} \|_{\infty}< \delta} |a_n|^{2n} d\ell_{\CC}(a_n)d\ell_{\CC^{n}}(z_1,\dots,z_n) \\ & = \int 1_{\mu_n \in A} \frac{\prod_{i<j}|z_i-z_j|^2}{\| \tilde{a} \|_{\infty}^{2n+2}}d\ell_{\CC^{n}}(z_1,\dots,z_n) \int |u|^{2n} 1_{|u|< \delta}d\ell_{\CC}(u) \\
  &=\frac{n!}{\pi^{n+1} }\int 1_{\mu_n \in A} \frac{\prod_{i<j}|z_i-z_j|^2}{\| \tilde{a} \|_{2}^{2n+2}}d\ell_{\CC^{n}}(z_1,\dots,z_n) \frac{\pi^{n+1} }{n!} \int |u|^{2n} 1_{|u|< \delta}d\ell_{\CC}(u) \\
& =\PP(\mu_n^{\CC} \in A) \frac{\pi^{n+1} }{n!} \int |u|^{2n} 1_{|u|< \delta}d\ell_{\CC}(u).
\end{align*}
Here we used the change of variables $u=\| \tilde{a} \|_{\infty} a_n$, using the fact that $\| \tilde{a} \|_{\infty}$ does not depend on $a_n$ and the inequality
 $\|.\|_{\infty} \leq \|.\|_2$ in
$\CC^{n+1}$.
Since
\[ \lim_{n \to \infty} \frac{1}{n^2} \log \left(\frac{\pi^{n+1} }{n!} \int |u|^{2n} 1_{|u|< \delta}d\ell_{\CC}(u)\right)  = 0, \]
we obtain
\[\liminf_{n \to \infty} \frac{1}{n^2} \log \left(\int 1_{\mu_n \in A}\prod_{i<j}|z_i-z_j|^2 \int 1_{|a_n|\| \tilde{a} \|_{\infty}< \delta} |a_n|^{2n} d\ell_{\CC}(a_n)d\ell_{\CC^{n}}(z_1,\dots,z_n)
\right)\geq -\inf_{\mathrm{int}A}I_{\CC}.\]
Combined with \eqref{eq-align} we obtain 
\[   \liminf_{n \to \infty} \frac{1}{n^2}\log \PP(\mu_n \in A)\geq -\varepsilon -\inf_{\mathrm{int}A}I_{\CC}.
\]
Taking the limit as $\varepsilon$ goes to zero completes
the proof of the lower bound.
\newline
\textbf{Real coefficients.} Let $E$ be $\RR$ or $\RR^+$.
The proof for real coefficients is essentially the same as for complex coefficients, except that the distribution of the roots is a mixture of measures instead of an absolutely continuous measure. We will apply the same ideas to each term of the mixture to obtain the upper and lower bound.
If the coefficients $a_k$'s are i.i.d. random variables with density $g$ with respect to the Lebesgue measure on $E$, then the distribution of the vector $(z_1,\dots,z_n,a_n)$ is given by:
\begin{align*}
\sum_{k=0}^{\lfloor n/2 \rfloor} \frac{2^k}{k!(n-2k)!} |a_n|^n & \prod_{i<j}|z_i-z_j| \prod_{k=0}^{n}g(a_i) d\ell_E(a_n)d\ell_{n,k}(z_1,\dots,z_n) \\ &= \sum_{k=0}^{\lfloor n/2 \rfloor}\frac{2^k}{k!(n-2k)!} p_{n,k}(z_1,\dots,z_n,a_n)d\ell_E(a_n)d\ell_{n,k}(z_1,\dots,z_n).
\end{align*}
Using exactly the same reasoning as in the complex case, we define $\theta_n^E$ as:
\begin{equation*}
\theta_n^E=\begin{cases}
\frac{\pi^{(n-1)/2}}{\Gamma(\frac{n+1}{2})} & \text{ if } E= \RR \\
\frac{1}{n!} & \text{ if } E= \RR^+
\end{cases}
\end{equation*}
and we notice that
\[\lim_{n \to \infty} \frac{1}{n^2}\log \theta_n^E = 0. \]
We obtain that for any $k$:
\[ \int p_{n,k}(z_1, \dots,z_n) d\ell_E(a_n) \leq  \frac{1}{\theta_n^E} \frac{ \prod_{i<j}|z_i-z_j|}{ (\int\prod_{i=1}^n |z-z_i|^2 d\nu_{S^1})^{(n+1)/2}} \theta_n^E.\]
This inequality implies that, for any Borel set $A \in \mathcal{ M}_1(\CC)$:
\begin{align*}
\PP(\mu_n \in A) & = \sum_{k=0}^{\lfloor n/2 \rfloor}\frac{2^k}{k!(n-2k)!} \int 1_{\mu_n \in A} \int p_{n,k}(z_1, \dots,z_n) d\ell_E(a_n) d\ell_{n,k}(z_1,\dots,z_n)\\
& \leq \theta_n^E \PP(\mu_n^E \in A).
\end{align*}
Using the large deviations principle for $(\mu_n^E)_{n \in \NN}$ ends the proof of the upper bound.

The proof of the lower bound is very similar to the complex case, we use the technical lemma to deal with i.i.d. uniform random variables on the disk $D(0,\delta)$.
\begin{align} \label{infreel}
\PP(\mu_n \in A) & = \int 1_{\mu_n \in A} \prod_{k=0}^n g(a_k) d\ell_E(a_0)\dots d\ell_E(a_n) \nonumber\\
& \geq \int 1_{\{\prod_{k=0}^{n}g(a_i)\geq e^{-\varepsilon n^2}\}}1_{\mu_n \in A}1_{\| a \|_{\infty}< \delta} \prod_{k=0}^n g(a_k) d\ell_E(a_0)\dots d\ell_E(a_n) \nonumber\\
& \geq  e^{-\varepsilon n^2} \int 1_{\mu_n \in A}1_{\| a \|_{\infty}< \delta}d\ell_E(a_0)\dots d\ell_E(a_n) - e^{-(K+\varepsilon)n^2}.
\end{align}
We transform this integral in order to compare it to one of the known cases.
\begin{align*}
\int 1_{\mu_n \in A}&1_{\| a \|_{\infty}< \delta} d\ell_E(a_0)\dots d\ell_E(a_n)\\ & = \sum_{k=0}^{\lfloor n/2 \rfloor}\frac{2^k}{k!(n-2k)!} \int 1_{\mu_n \in A}1_{\| a \|_{\infty}< \delta} |a_n|^n \prod_{i<j}|z_i-z_j|d\ell_E(a_n)d\ell_{n,k}(z_1,\dots,z_n) \\
& = \sum_{k=0}^{\lfloor n/2 \rfloor}\frac{2^k}{k!(n-2k)!} \int |u|^{n+1} 1_{|u|<\delta} d\ell_E(u)  \int 1_{\mu_n \in A}\frac{\prod_{i<j}|z_i-z_j|}{\| \tilde{a} \|_{\infty}^{n+1}} d\ell_{n,k}(z_1,\dots,z_n).
\end{align*}
If $E=\RR$, we use the inequality 
$ \|. \|_{\infty} \leq \|. \|_{2}$ 
on $\RR^{n+1}$
to obtain
\begin{align*}
\int & 1_{\mu_n \in A}  1_{\| a \|_{\infty}< \delta} d\ell_{\RR}(a_0)\dots d\ell_{\RR}(a_n) \\ & \geq \sum_{k=0}^{\lfloor n/2 \rfloor}\frac{2^k}{k!(n-2k)!} \int |u|^{n+1} 1_{|u|<\delta} d\ell_{\RR}(u)  \int 1_{\mu_n \in A}\frac{\prod_{i<j}|z_i-z_j|}{\| \tilde{a} \|_{2}^{n+1}} d\ell_{n,k}(z_1,\dots,z_n).
\end{align*} 
If $E=\RR^+$, we use the 
inequality $ \|. \|_{\infty} \leq \|. \|_{1}$ on $\RR^{n+1}$ to obtain
\begin{align*}
\int &1_{\mu_n \in A}  1_{\| a \|_{\infty}< \delta} d\ell_{\RR^+}(a_0)\dots d\ell_{\RR^+}(a_n) \\ & \geq \sum_{k=0}^{\lfloor n/2 \rfloor}\frac{2^k}{k!(n-2k)!} \int |u|^{n+1} 1_{|u|<\delta} d\ell_{\RR^+}(u)  \int 1_{\mu_n \in A}\frac{\prod_{i<j}|z_i-z_j|}{\| \tilde{a} \|_{1}^{n+1}} d\ell_{n,k}(z_1,\dots,z_n).
\end{align*} 
Note that the only difference between the case $\RR$ and the case $\RR^+$ 
is the reference norm employed.
Using \eqref{infreel} with the last two inequalities, we obtain that for any $\varepsilon>0$ fixed, we have:
\begin{align*}
\liminf_{n \to \infty} \frac{1}{n^2}\log \PP(\mu_n \in A) &\geq \liminf_{n \to \infty} \frac{1}{n^2}\log \PP(\mu_n^E \in A)+ \lim_{n \to \infty} \frac{1}{n^2}\log \theta_n^E - \varepsilon \\
&\geq -\inf_{\mathrm{int}A} I_E -\varepsilon.
\end{align*}
Taking the limit as $\varepsilon$ goes to zero ends the proof of the large deviations lower bound for the real and positive cases.
\end{proof}

\section{Concluding remarks and an open problem.}
We focused in this note
on Kac polynomials but we could as well 
study the universality of the large deviations for the zeros of
\[P_n=\sum_{k=0}^{n}a_k R_k\]
where the $R_k$'s are orthogonal polynomials satisfying the assumptions of regularity given in \cite{zeitounizelditch} and \cite{butez}. In this case, the distribution of the zeros can be computed (\cite[Theorem 5.1]{butez}) and, under the same hypotheses as in Theorem \ref{maintheorem}, the same large deviations principle as for Gaussian coefficients holds. Similar ideas apply to certain 
non i.i.d. models such as the $P(\phi)_2$ model of \cite{feng}.

A significant limitation of our approach is the use of the assumption
\eqref{inf} in Theorem \ref{maintheorem}. While it is possible that it can
be relaxed, we note that some assumption of this type is necessary for the 
universality resut. Indeed,
if the support of the distribution of the coefficients is inside an annulus, 
it follows from Jensen's formula, see
\cite{hughes}, that
$\mu_n$ converges deterministically towards $\nu_{S^1}$. Hence, no non-trivial 
LDP can hold in this case. 
An interesting test case is the case where the i.i.d. coefficients 
possess the density $|z|^{\alpha}1_{|z|<\delta}$ for some $\alpha >0$ and 
$\delta >0$. In that case,
the distribution of the zeros $(z_1,\dots,z_n)$ is absolutely continuous 
with respect to the Lebesgue measure on $\CC^n$ with density proportional to:
\begin{align*}
\prod_{i<j}|z_i-z_j|^2 \int |a_n|^{2n} \prod|a_i|^{\alpha} 1_{\|a\|_{\infty}< \delta} d\ell_{\CC}(a_n)=  \frac{\prod_{i<j}|z_i-z_j|^2}{\|\tilde{a}\|_{\infty}^{2n+2+n\alpha}} \prod_{k=0}^n \frac{|a_k|}{|a_n|}.
\end{align*}
If we are able to prove that the term 
$\prod_{k=0}^n \frac{|a_k|}{|a_n|}$ does not contribute to the 
large deviations, then a LDP at speed $n^2$ would follow
with rate function
\[I_{\alpha}(\mu)=  -\iint \log|z-w|d\mu(z)d\mu(w) + (2+\alpha)\sup_{z \in S^1} \int \log|z-w|d\mu(w).\] 
In particular, we do not expect universality in that case. We have not been able to carry out the analysis of this setup.

\bibliographystyle{alpha} 
\bibliography{biblio} 

\end{document}